\documentclass[a4paper]{amsart} 

\usepackage{amsmath,amsthm,amssymb,amsfonts,mathrsfs,color,hyperref, mathtools,crop, graphicx, enumitem, todonotes}
\usepackage[]{geometry}
\usepackage{ esint }
\newcommand\showlabel{\addtocounter{equation}{1}\tag{\theequation}}

\theoremstyle{plain}
\begingroup	
\newtheorem{theorem}{Theorem}[section]
\newtheorem*{theorem*}{Theorem}
\newtheorem*{"theorem"}{``Theorem''}
\newtheorem{corollary}[theorem]{Corollary}

\newtheorem{lemma}[theorem]{Lemma}
\endgroup

\theoremstyle{definition}
\begingroup
\newtheorem{definition}[theorem]{Definition}
\endgroup

\theoremstyle{remark}
\begingroup
\newtheorem{remark}[theorem]{Remark}
\newtheorem{example}[theorem]{Example}
\endgroup 

\numberwithin{equation}{section}
\newcommand{\N}{\mathbb N} 
 
\newcommand{\Z}{\mathbb Z} 
\newcommand{\R}{\mathbb R} 
\newcommand{\dist}{{\rm dist}}
\newcommand{\diam}{{\rm diam}}

\newcommand{\spt}{{\mathrm{spt}}}
\newcommand{\id}{\mathrm{id}}

\renewcommand{\H}{{\mathcal H}}
\newcommand{\E}{{\mathbb E}}
\newcommand{\W}{{\mathcal W}}
\newcommand{\M}{{\mathcal M}}

\renewcommand{\L}{{\mathcal L}}

\newcommand{\F}{{\mathcal F}}
\newcommand{\G}{{\mathcal G}}

\newcommand{\B}{{\mathcal B}}

\newcommand{\LRa} {\Leftrightarrow}
\newcommand{\Ra} {\Rightarrow}

\newcommand{\embeds}{\xhookrightarrow{\quad}}

\renewcommand{\d}{\mathrm{d}}

\newcommand{\dx}{\,\mathrm{d}x}
\newcommand{\dy}{\,\mathrm{d}y}

\newcommand{\dt}{\,\mathrm{d}t}

\let \ol = \overline

\newcommand{\eps}{\varepsilon}
\newcommand{\average}{{\mathchoice {\kern1ex\vcenter{\hrule height.4pt
width 6pt depth0pt} \kern-9.7pt} {\kern1ex\vcenter{\hrule
height.4pt width 4.3pt depth0pt} \kern-7pt} {} {} }}

\allowdisplaybreaks
\DeclareMathOperator{\argmin}{argmin}

\renewcommand{\P}{\mathbb{P}}

 \makeatletter
\@namedef{subjclassname@2020}{2020 Mathematics Subject Classification}
\makeatother

\newcommand{\cor}[1]{\textcolor{black}{#1}}

\begin{document}

\title[Kolmogorov Width: Shallow Networks, Random Features, NTK]{Kolmogorov Width Decay and Poor Approximators in Machine Learning: Shallow Neural Networks, Random Feature Models and Neural Tangent Kernels}

\author{Weinan E}
\address{Weinan E\\
Department of Mathematics and Program in Applied and Computational Mathematics\\ Princeton University\\ Princeton, NJ 08544\\ USA}
\email{weinan@math.princeton.edu}

\author{Stephan Wojtowytsch}
\address{Stephan Wojtowytsch\\
Program in Applied and Computational Mathematics\\
Princeton University\\
Princeton, NJ 08544
}
\email{stephanw@princeton.edu}

\date{\today}

\dedicatory{Dedicated to Andrew Majda on the occasion of his 70th birthday}


\begin{abstract}
We establish a scale separation of Kolmogorov width type between subspaces of a given Banach space under the condition that a sequence of linear maps converges much faster on one of the subspaces. The general technique is then applied to show that reproducing kernel Hilbert spaces are poor $L^2$-approximators for the class of two-layer neural networks in high dimension, and that \cor{multi-layer} networks with small path norm are poor approximators for certain Lipschitz functions, also in the $L^2$-topology.
\end{abstract}

\maketitle


\section{Introduction}

It has been known since the early 1990s that two-layer neural networks with sigmoidal or ReLU activation can approximate arbitrary continuous functions on compact sets in the uniform topology \cite{cybenko1989approximation,hornik1991approximation}. In fact, when approximating a suitable (infinite-dimensional) class of functions in the $L^2$ topology of {\em any} compactly supported Radon probability measure, two-layer networks can evade the curse of dimensionality \cite{barron1993universal}. In this article, we show that

\begin{enumerate}
\item infinitely wide random feature functions with norm bounds are much worse approximators in high dimension compared to two-layer neural networks.
\item infinitely wide neural networks are subject to the curse of dimensionality when approximating general Lipschitz functions in high dimension.
\end{enumerate}

In both cases, we consider approximation in the $L^2([0,1]^d)$-topology. \cor{The second statement applies more generally to any model in which few data samples are needed to estimate integrals uniformly over the hypothesis class.} In the first point, we can consider more general kernel methods instead of random features (including certain neural tangent kernels), and the second claim \cor{holds true for multi-layer networks as well as deep ResNets of bounded width}. We conjecture that Lipschitz functions in the second statement could be replaced with $C^k$ functions for fixed $k$. Precise statements of the results are given in Corollary \ref{corollary l2 slow} and Example \ref{example rkhs slow}.

To prove these results, we show more generally that if $X, Y$ are subspaces of a Banach space $Z$ and a sequence of linear maps $A_n$ converges quickly to a limit $A$ on $X$, but not on $Y$, then there must be a Kolmogorov width-type separation between $X$ and $Y$. The classical notion of Kolmogorov width is considered in Lemma \ref{lemma slow sequence} and later extended to a stronger notion of separation in Lemma \ref{lemma slow point}.
 
We apply the abstract result to the pairs $X=$ Barron space (for two-layer networks) \cor{or $X=$ tree-like function space (for multi-layer networks)}/$Y=$ Lipschitz space, and $X=$ RKHS/$Y=$ Barron space. In the first case, the sequence of linear maps is given by a type of Monte-Carlo integration, in the second case by projection onto the eigenspaces of the RKHS kernel.
 
This article is structured as follows. In Section \ref{section abstract}, we prove the abstract result which we apply to \cor{Barron/tree-like function spaces} and Lipschitz space in Section \ref{section concrete} and to RKHS and Barron space in Section \ref{section rkhs}. We conclude by discussing our results and some open questions in Section \ref{section discussion}. In appendices \ref{appendix barron} and \ref{appendix rkhs}, we review the natural function spaces for shallow neural networks and kernel methods respectively. In Appendix \ref{appendix rkhs}, we specifically focus on kernels arising from random feature models and neural tangent kernels for two-layer neural networks.

\subsection{Notation}

We denote the closed ball of radius $r>0$ around the origin in a Banach space by $B^X_r$ and the unit ball by $B^X_1 = B^X$. The space of continuous linear maps between Banach spaces $X,Y$ is denoted by $L(X,Y)$ and the continuous dual space of $X$ by $X^* = L(X,\R)$. \cor{The support of a Radon measure $\mu$ is denoted by $\spt\,\mu$.}

\section{An Abstract Lemma}\label{section abstract}

\subsection{Kolmogorov Width Version}

The Kolmogorov width of a function class $\F$ in another function class $\G$ with respect to a metric $d$ on the union of both classes is defined as the biggest distance of an element in $\G$ from the class $\F$:
\[
w_d(\F;\G) = \sup_{g\in \G} \dist(g, \F) = \sup_{g\in \G}\inf_{f\in \F} d(f,g).
\]
In this article, we consider the case where $\G$ is the unit ball in a Banach space $Y$, $\F$ is the ball of radius $t>0$ in a Banach space $X$ and $d= d_Z$ is induced by the norm on a Banach space $Z$ into which both $X$ and $Y$ embed densely. As $t$ increases, points in $Y$ are approximated to higher degrees of accuracy by elements of $X$. The rate of decay
\[
\rho(t) := w_{\cor{d_Z}}(B^X_t, B^Y_1) 
\]
provides a quantitative measure of density of $X$ in $Y$ with respect to the topology of $Z$. For a different point of view on width focusing on approximation by finite-dimensional spaces, see \cite[Chapter 9]{lorentz1966approximation}.

In the following Lemma, we show that if there exists a sequence of linear operators on $Z$ which behaves sufficiently differently on $X$ and $Y$, then $\rho$ must decay slowly as $t\to\infty$.

\begin{lemma}\label{lemma slow sequence}
Let $X, Y, Z, W$ be Banach spaces such that $X, Y\embeds Z$. Assume that $A_n, A : Z\to W$ are continuous linear operators such that
\[
 \|A_n - A\|_{L(X,W)} \leq C_X\, n^{-\alpha}, \qquad \|A_n-A\|_{L(Y,W)} \geq c_Y\, n^{-\beta}, \qquad \|A_n - A\|_{L(Z,W)}\leq C_Z
\]
for $\beta<\alpha$ and constants $C_X, c_Y, C_Z>0$. Then
\begin{equation}
\rho(t) \geq 2^{-\beta} \frac{\left( c_Y/2\right)^\frac{\alpha}{\alpha-\beta}}{C_Z\,C_X^\frac{\beta}{\alpha-\beta}}\,t^{-\frac\beta{\alpha-\beta}}\quad\forall\ t\geq \frac{c_Y}{2C_X}
\end{equation}
and
\begin{equation}
\liminf_{t\to\infty} \left(t^\frac\beta{\alpha-\beta}\: \rho(t) \right) \geq \frac{\left( c_Y/2\right)^\frac{\alpha}{\alpha-\beta}}{C_Z\,C_X^\frac{\beta}{\alpha-\beta}}.
\end{equation}
\end{lemma}

\begin{proof}
Choose a sequence $y_n\in  B^Y$ such $\|(A_n-A)y_n\|_{W} \geq c_Y\,n^{-\beta}$ and $x_n\in X$ such that
\[
x_n \in \argmin_{\{x : \|x\|_X\leq t_n\}} \|x - y_n\|_Z\qquad\text{for }\quad t_n:= \frac{c_Y}{2\,C_X}\,n^{\alpha-\beta}
\]
(see Remark \ref{remark minimizer}). Then
\begin{align*}
c_Y\,n^{-\beta} &\leq \|(A_n-A)y_n\|_W\\
	&\leq \|(A_n-A)(y_n-x_n)\|_W+ \|(A_n-A)x_n\|_W\\
	&\leq C_Z\,\|y_n-x_n\|_Z  + C_X\,n^{-\alpha}\|x_n\|_X\\
	&\leq C_Z\,\|x_n - y_n\|_Z + \frac{c_Y}2\,{n^{-\beta}}.
\end{align*}
We therefore have
\[
\|x_n-y_n\|_Z \geq \frac{c_Y}{2\,C_Z}\,n^{-\beta}  = \frac{c_Y}{2\,C_Z} \left(\frac{2\,C_X}{c_Y} \, t_n\right)^\frac{-\beta}{\alpha-\beta} =\left( \frac{c_Y}{2}\right)^\frac{\alpha}{\alpha-\beta}\frac{1}{C_Z\,C_X^\frac{\beta}{\alpha-\beta}}\:t_n^{-\frac\beta{\alpha-\beta}}.
\]
Clearly $t_n\to \infty$ since $\alpha>\beta$. For general $t>0$, take $t_n = \inf \{t_k: k\in \N, t_k\geq t\}$. Then
\begin{align*}
\rho(t) &\geq\rho(t_n)\\
	 &\geq \frac{\left( c_Y/2\right)^\frac{\alpha}{\alpha-\beta}}{C_Z\,C_X^\frac{\beta}{\alpha-\beta}} \:t_n^{-\frac\beta{\alpha-\beta}}\\
	 &\geq \frac{\left( c_Y/2\right)^\frac{\alpha}{\alpha-\beta}}{C_Z\,C_X^\frac{\beta}{\alpha-\beta}} \left(\frac{t_{n-1}}{t_n}\right)^\frac{\beta}{\alpha-\beta} \,t^{-\frac\beta{\alpha-\beta}}\\
	 &= \frac{\left( c_Y/2\right)^\frac{\alpha}{\alpha-\beta}}{C_Z\,C_X^\frac{\beta}{\alpha-\beta}} \left(\frac{n-1}{n}\right)^{\beta}t^{-\frac\beta{\alpha-\beta}}.
\end{align*}
As $t\to\infty$, so does $n$, and the $n$-dependent term converges to $1$.
\end{proof}

\begin{remark}\label{remark minimizer}
Generally elements like $x_n, y_n$ may not exist if the extremum is not attained. Otherwise, we can choose $x_n$ such that $\|x_n - y_n\|_Z$ is sufficiently close to its infimum and $\|(A_n-A) y_n\|$ is sufficiently close to its supremum. To simplify our presentation, we assume that the supremum and infimum are attained.

The choice of $x_n$ as a minimizer is valid if 
\begin{enumerate}
\item $X$ embeds into $Z$ compactly, so the minimum of the continuous function $\|\cdot - y\|_Z$ is attained on the compact set $\{\|\cdot\|_X\leq t_k\}$, or 
\item the embedding $X\embeds Z$ maps closed bounded sets to closed sets and $Z$ admits continuous projections onto closed convex sets (for example, $Z$ is uniformly convex).
\end{enumerate}
In the applications below, the first condition will be met. 
\end{remark}

\subsection{Improved Estimate}

In the previous section, we have shown by elementary means that the estimate
\[
\liminf_{t\to\infty} \left(t^\gamma \sup_{\|y\|\leq 1} \inf_{\|x\|_X\leq t} \|x-y\|_Z\right) \geq c>0
\]
holds for suitable $\gamma$ if a sequence of linear maps between $Z$ and another Banach space $W$ behaves very differently on subspaces $X$ and $Y$ of $Z$. So intuitively, on each scale $t>0$ there exists an element $y_t\in B^Y$ such that $y_t$ is poorly approximable by elements in $X$ on this scale. In this section, we establish that there exists a single point $y\in Y$ which is poorly approximable across infinitely many scales. This statement has applications in Wasserstein gradient flows for machine learning which we discuss in a companion article \cite{dynamic_cod}.

\begin{lemma}\label{lemma slow point}
Let $X, Y, Z$ be Banach spaces such that $X, Y\embeds Z$. Assume that $A_n, A\in L(Z, W)$ are operators such that
\[
\|A_n - A\|_{L(X,W)} \leq C_X\,n^{-\alpha}, \qquad \|A_n-A\|_{L(Y,W)} \geq c_Y\,n^{-\beta}, \qquad \|A_n - A\|_{L(Z,W)} \leq C_Z
\]
for $\beta<\frac\alpha2$ and constants $C_X, c_Y, C_Z$. Then there exists $y\in B^Y$ such that for every $\gamma > \frac{\beta}{\alpha-\beta}$ we have
\[
\limsup_{t\to\infty} \left(t^\gamma \inf_{\|x\|_X\leq t} \|x-y\|_Z\right) = \infty.
\]
\end{lemma}

The result is stronger than the previous one in that it fixes a single point $y$ which is poorly approximable in infinitely many scales $t_{n_k}$. While in each scale $t_n$ there exists a point $y_n$ which is poorly approximable, we only show that $y$ is poorly approximable in infinitely many scales, not in all scales. 

\begin{proof}[Proof of Lemma \ref{lemma slow point}]
Since $Y\embeds Z$, there exists a constant $C^Y>0$ such that $\|A_n - A\|_{Y^*}\leq C^Y$.

{\bf Definition of $y$.}
Choose sequences $y_n\in  B^Y$ and $w_n^* \in B^{W^*}$ such 
\[
w_n^*\circ (A_n-A)(y_n)\geq c_Y\,n^{-\beta}.
\]
Consider two sequences $n_k, m_k$ of strictly increasing integers such that
\[
\sum_{k=1}^\infty \frac1{n_k}\leq 1.
\]
We will impose further conditions below. Set
\[
y:= \sum_{k=0}^\infty \frac{\eps_k}{n_k}\,y_{m_k}
\]
where the signs $\eps_k\in\{-1,1\}$ are chosen inductively such that
\[
\eps_K\cdot w_{m_k}^*\circ (A_{m_K}-A)\left(\sum_{k=1}^{K-1}\frac{\eps_k}{n_k}y_{m_k}\right) \geq 0.
\]
Clearly
\[
\|y\|_Y \leq \sum_{k=1}^\infty \frac{|\eps_k|}{n_k}\,\|y_{m_k}\|_Y = \sum_{k=1}^\infty\frac1{n_k}=1.
\]
To shorten notation, define $L_k = w_{m_k}^* \circ (A_{m_k} - A) \in Z^*$ and note that the estimates for $A_{m_k} - A$ transfer to $L_k$.
If $\eps_K=1$ we have
\begin{align*}
L_ky &= L_k\left(\sum_{k=1}^{K-1}\frac{\eps_k}{n_k}y_{m_k}\right) + \frac1{n_K}\,L_k\,y_{m_K} + L_k\left(\sum_{k=K+1}^{\infty}\frac{\eps_k}{n_k}y_{m_k}\right)\\
	&\geq 0 +\frac1{n_K}\,L_ky_{m_K} - C^Y\sum_{l= K+1}^\infty \frac1{n_l}\\
	&\geq \frac1{n_K} \left(c_Y\,m_K^{-\beta} - C^Y\,n_K \sum_{l=k+1}^\infty \frac{1}{n_l}\right)
\end{align*}
\cor{where the infinite tail of the series is estimated by $\|L_k\|_{L(Y,W)}\leq C^Y$ and $\|y_{m_k}\|_Y\leq 1$. S}imilarly if $\eps_K=-1$ we obtain
\[
L_ky \leq - \frac1{n_K} \left(c_Y\,m_K^{-\beta} - C^Y\,n_K \sum_{l=K+1}^\infty \frac{1}{n_l}\right).
\]
{\bf Slow approximation rate.} Choose 
\[
t_k:= \frac{c_Y\,m_k^{\alpha-\beta}}{2\,C_X\,n_k},\qquad x_k \in \argmin_{\|x\|_X \leq t_k} \|x - y\|_Z.
\]

Then
\begin{align*}
\frac1{n_k} \left(c_Y\,m_k^{-\beta} - C^Y\,n_k \sum_{l=k+1}^\infty \frac{1}{n_l}\right)&\leq \big|L_ky\big|\\
	&\leq \big| L_k(y-x_k)\big| + \big|L_kx_k\big|\\
	&\leq C_Z\,\|y-x_k\|_Z  +\|A_{m_k}-A\|_{X^*}\,\|x_k\|_X\\
	&\leq C_Z\,\|y-x_k\|_Z + C_X\,m_k^{-\alpha}\,t_k.
\end{align*}
Since $t_k$ was chosen precisely such that 
\[
C_X m_k^{-\alpha}t_k = \frac{c_Y}{2\,n_k}\,m_k^{-\beta}, 
\] we obtain that
\begin{equation}\label{eq lower estimate}
\frac1{2C_Z\,n_k} \left(c_Y\,m_k^{-\beta} - 2C^Y\,n_k \sum_{l=k+1}^\infty \frac{1}{n_l}\right) \leq \|y-x_k\|_Z = \min_{\|x\|_X \leq t_k} \|x- y\|_Z.
\end{equation}
For this lower bound to be meaningful, the first term in the bracket has to dominate the second term.
We specify the scaling relationship between $n_k$ and $m_k$ as 
\[
m_k = n_k^\frac{ k }{\alpha-\beta}.
\]
In this definition, $m_k$ is not typically an integer unless $\frac1{\alpha-\beta}$ is an integer (or, to hold for a subsequence, rational). In the general case, we choose the integer $\tilde m_k$ closest to $m_k$. To simplify the presentation, we proceed with the non-integer $m_k$ and note that the results are insensitive to perturbations of order $1$. 

We obtain
\[
 t_k = \frac{c_Y}{2\,C_X} \frac{m_k^{\alpha-\beta}}{n_k} = \frac{c_Y}{2C_X}\,n_k^{ k -1}, \qquad \frac{m_k^{-\beta}}{n_k} = n_k^{-\frac{\beta k }{\alpha-\beta}-1} = n_k^{-\frac{\beta(k-1)  +\alpha}{\alpha-\beta}} 
=\left(\frac{2C_X}{c_Y}\,t_k\right)^{-\frac\beta{\alpha-\beta} -\frac{\alpha}{( k -1)(\alpha-\beta)}}.
\]
In particular, note that $t_k\to \infty$ as $k\to\infty$. In order for 
\[
n_k \sum_{l=k+1}^\infty \frac1{n_l}
\]
to be small, we need $n_k$ to grow super-exponentially. Note that $\frac{\beta}{\alpha-\beta} \leq 1$ since $\beta\leq \frac\alpha2$.
We specify $n_k = 2^{(k^k)}$ and compute
\begin{align*}
 \sum_{l=k+1}^\infty \frac1{n_l} &= \sum_{l=1}^\infty 2^{-\big((k+l)^k\,( k +l)^l\big)}\quad
 	 \leq \sum_{l=1}^\infty 2^{-(k^k\,(k+l)^l)}\quad
	 = \sum_{l=1}^\infty \left(\frac1{n_k}\right)^{\big((k+l)^l\big)}\\
	& \leq \frac 2{n_k^{k+1}}\quad
	 \ll n_k^{-\frac{\beta k }{\alpha-\beta}-1}\quad
	 = \frac{m_k^{-\beta}}{n_k}
\end{align*}
for large enough $k$. Thus we can neglect the negative term on the left hand side of \eqref{eq lower estimate} at the price of a slightly smaller constant. Thus 
\[
\frac{c_Y}{4C_Z}\,\left(\frac{2C_X}{c_Y}\,t_k\right)^{-\frac\beta{\alpha-\beta} -\frac{\alpha}{( k -1)(\alpha-\beta)}} = \frac{c_Y}{4C_Z\,n_k} \,m_k^{-\beta} \leq \min_{\|x\|_X \leq t_k} \|x - y\|_Z.
\]
Finally, we conclude that for all $\gamma>\frac{\beta}{\alpha-\beta}$ we have
\begin{align*}
\limsup_{t\to\infty}\left(t^\gamma \,\inf_{\|x\|_X\leq t\}} \|x-y\|_Z\right) &\geq \limsup_{k\to\infty} \left(t_k^\gamma \,\inf_{\|x\|_X\leq t_k\}} \|x-y\|_Z\right) \\
	&\geq C_{X,Y,Z}\,\limsup_{k\to\infty} t_k^{\gamma -\frac\beta{\alpha-\beta} - \frac{\alpha}{( k -1)(\alpha-\beta)}}\\
	&= \infty.
\end{align*}
\end{proof}

\section{Approximating Lipschitz Functions by Functions of Low Complexity}\label{section concrete}

In this section, we apply Lemma \ref{lemma slow point} to the situation where general Lipschitz functions are approximated by functions in a space with much lower complexity. Examples include function spaces for infinitely wide neural networks with a single hidden layer and spaces for deep ResNets of bounded width. For simplicity, we first consider uniform approximation and then modify the ideas to also cover $L^2$-approximation.

\subsection{Approximation in $L^\infty$}
Consider the case where 
\begin{enumerate}
\item $Z$ is the space of continuous functions on the unit cube $Q=[0,1]^d\subset \R^d$ with the norm
\[
\|\phi\|_Z = \sup_{x\in Q} \phi(x), 
\]
\item $Y$ is the space of Lipschitz-continuous functions with the norm
\[
\|\phi\|_Y = \sup_{x\in Q} \phi(x) + \sup_{x\neq y} \frac{|\phi(x) - \phi(y)|}{|x-y|}, \text{ and}
\]
\item $X$ is a Banach space of functions such that
\begin{itemize}
\item $X$ embeds continuously into $Z$,
\item the Monte-Carlo estimate 
\[
\E_{X_i \sim \L^d|_Q\text{ iid}} \left\{\sup_{\phi \in B^X} \left[\frac1n\sum_{i=1}^n\phi(X_i) - \int_Q\phi(x)\,\dx\right]\right\} \leq  \frac{C_X}{\sqrt{n}}
\]
holds. \cor{Here and in the following when no other measure is specified, we assume that integrals are taken with respect to Lebesgue measure.}
\end{itemize}
\end{enumerate}

Examples of admissible spaces for $X$ are Barron space for two-layer ReLU networks and the compositional function space for deep ReLU ResNets of finite width, see \cite{weinan2019priori,E:2018ab,weinan2019lei}. A brief review of Barron space is provided in Appendix \ref{appendix barron}. The Monte-Carlo estimate is proved by estimating the Rademacher complexity of the unit ball in the respective function space. For Barron space, $C_X = 2\sqrt{2\,\log(2d)}$ and for compositional function space $C_X = e^2\,\sqrt{\log(2d)}$, see \cite[Theorems 6 and 12]{weinan2019lei}.

We observe the following: If $X$ is a vector of iid random variables sampled from the uniform distribution on $Q$, then
\begin{align*}
	\sup_{\phi \text{ is 1-Lipschitz}} \left(\frac1n\sum_{i=1}^n\phi(X_i) - \int_Q\phi(x)\,\dx\right)
	&= W_1\left(\L^d|_Q, \:\frac1n\sum_{i=1}^n\delta_{X_i}\right) 
\end{align*}
is the $1$-Wasserstein distance between $d$-dimensional Lebesgue measure on the cube and the empirical measure generated by the random points -- see \cite[Chapter 5]{villani2008optimal} for further details on Wasserstein distances and the link between Lipschitz functions and optimal transport theory. The distance on $\R^d$ for which the Wasserstein transportation cost is computed is the same for which $\phi$ is $1$-Lipschitz.

Empirical measures converge to the underlying distribution slowly in high dimension \cite{fournier2015rate}, by which we mean that
\[
\E_{X\sim \big(\L^d|_Q\big)^n}W_1\left(\L^d|_Q,\: \frac1n\sum_{i=1}^n\delta_{X_i}\right) \geq c_d\,n^{-1/d}
\]
for some dimension-dependent constant $d$. Observe that also
\begin{align*}
\sup_{\phi \text{ is 1-Lipschitz}} \left(\frac1n\sum_{i=1}^n\phi(X_i) - \int_Q\phi(x)\,\dx\right)
	&= \sup_{\phi \text{ is 1-Lipschitz}, \:\phi(0) = 0} \left(\frac1n\sum_{i=1}^n\phi(X_i) - \int_Q\phi(x)\,\dx\right)\\
	&\leq \big[1+\diam(Q)\big]\, \sup_{\|\phi\|_{Y}\leq 1,\: \phi(0) = 0} \left(\frac1n\sum_{i=1}^n\phi(X_i) - \int_Q\phi(x)\,\dx\right)\\
	&\leq \big[1 + \diam(Q)\big] \, \sup_{\|\phi\|_{Y}\leq 1} \left(\frac1n\sum_{i=1}^n\phi(X_i) - \int_Q\phi(x)\,\dx\right)
\end{align*}
where $\diam(Q)$ is a diameter of the $d$-dimensional unit cube with respect to the norm for which $\phi$ is $1$-Lipschitz. Here we used that  replacing $\phi$ by $\phi + c$ for $c\in \R$ does not change the difference of the two expectations, and that on the space of functions with $\phi(0) = 0$ the equivalence
\[
[\phi]_Y := \sup_{x\neq y} \frac{|\phi(x) - \phi(y)|}{|x-y|} \leq \|\phi\|_Y \leq \big(1 + \diam(Q)\big)\,[\phi]_Y
\]
holds. By $\omega_d$ we denote the Lebesgue measure of the unit ball in $\R^d$ with respect to the correct norm.

\begin{lemma}
For every $n\in\N$ we can choose $n$ points $x_1,\dots, x_n$ in $Q$ such that
\[
\sup_{\phi \in B^Y} \left[\frac1n\sum_{i=1}^n\phi(x_i) - \int_Q\phi(x)\,\dx\right] \geq \frac d{d+1}\,\frac1{\big[(d+1)\,\omega_d\big]^{\frac1d}}\, \frac{1}{1 + \diam(Q)}\,\,n^{-1/d}
\]
and
\[
\sup_{\phi \in B^X} \left[\frac1n\sum_{i=1}^n\phi(x_i) - \int_Q\phi(x)\,\dx\right] \leq \frac{C_X}{n^{1/2}}.
\]
\end{lemma}

\begin{proof}
First, we prove the following. {\em Claim:} Let $x_1,\dots,x_n$ be any collection of $n$ points in $Q$. Then 
\[
 W_1\left(\L^d|_Q, \: \frac1n\sum_{i=1}^n\delta_{x_i}\right) \geq \frac d{d+1}\,\frac1{\big[(d+1)\,\omega_d\big]^{\frac1d}}\,n^{-1/d}.
\]

{\em Proof of claim:} Choose $\eps>0$ and consider the set
\[
U = \bigcup_{i=1}^n B_{\eps\,n^{-1/d}}(x_i).
\]
We observe that 
\[
\L^d(U\cap Q) \leq \L^d(U) \leq \sum_{i=1}^n \L^d\big(B_{\eps\,n^{-1/d}}(x_i)\big) = n\,\omega_d\,\big(\eps\,n^{-1/d}\big)^d = \omega_d\,\eps^d.
\]
So any transport plan between $\L^d|_Q$ and the empirical measure needs to transport mass $\geq 1-\omega_d\,\eps^d$ by a distance of at least $\eps\,n^{-1/d}$. We conclude that
\[
W_1\left(\L^d|_Q,\: \frac1n\sum_{i=1}^n\delta_{x_i}\right) \geq\sup_{\eps\in (0,1)} \big(1-\omega_d\,\eps^d\big)\eps\,n^{-1/d}.
\]
The infimum is attained when
\[
0 = 1 - (d+1) \omega_d \eps^d\quad  \LRa \quad \eps = \big[(d+1)\,\omega_d\big]^{-\frac1d} \quad \Ra\quad 1-\omega_d\eps^d = 1 - \frac1{d+1} = \frac d{d+1}.
\]
This concludes the proof of the claim.

{\em Proof of the Lemma:} Using the claim, any $n$ points $x_1,\dots,x_n$ such that 
\[
\sup_{\phi \in B^X} \left[\frac1N\sum_{i=1}^N\phi(x_i) - \int_Q\phi(x)\,\dx\right] \leq \E \left\{\sup_{\phi \in B^X} \left[\frac1N\sum_{i=1}^N\phi(X_i) - \int_Q\phi(x)\,\dx\right]\right\}  \leq \frac{C_X}{n^{1/2}}  
\]
satisfy the conditions.
\end{proof}

For any $n$, we fix such a collection of points $x_1^n,\dots, x^n_n$ and define
\[
A_n: Z\to \R, \quad A_n(\phi) = \frac1n\sum_{i=1}^n \phi(x^n_i), \qquad A:Z\to\R, \quad A(\phi) = \int_Q \phi(x)\dx.
\]
Clearly 
\[
|A\phi|, \:|A_n\phi| \leq \|\phi\|_{C^0} = \|\phi\|_Z.
\]
Thus we can apply Lemma \ref{lemma slow point} with
\[
\frac\beta{\alpha-\beta}= \frac{\frac1d}{\frac12-\frac1d} = \frac1{d\,\frac{d-2}{2d}} = \frac2{d-2}.
\]

\begin{corollary}
There exists a $1$-Lipschitz function $\phi$ on $Q$ such that 
\[
\limsup_{t\to\infty} \left(t^\gamma\,\inf_{\|f\|_X \leq t}\|\phi - f\|_{L^\infty(Q)}\right) = \infty.
\]
for all $\gamma> \frac{2}{d-2}$.
\end{corollary}

\subsection{Approximation in $L^2$}

Point evaluation functionals are no longer well defined if we choose $Z= L^2(Q)$. We therefore need to replace $A_n$ by functionals of the type
\[
A_n(\phi) = \frac1n\sum_{i=1}^n \fint_{B_{\eps_n}(X_i^n)}\phi\dx
\]
for sample points $X_i^n$ and find a balance between the radii $\eps_n$ shrinking too fast (causing the norms $\|A_n\|_{Z^*}$ to blow up) and $\eps_n$ shrinking too slowly (leading to better approximation properties on Lipschitz functions). 

We interpret $Q$ as the unit cube for function spaces, but as a $d$-dimensional flat torus when considering balls. Namely the ball $B_\eps(x)$ in $Q$ is to be understood as projection of the ball of radius $\eps>0$ around $[x]$ on $\R^d/ \Z^d$ onto $Q$. This allows us to avoid boundary effects. 

\begin{lemma}\label{lemma L2 approximation}
For every $n\in\N$ we can choose $n$ points $x_1,\dots, x_n$ in $Q$ such that the estimates
\begin{align*}
\sup_{\phi \in B^X} \left[\frac1n\sum_{i=1}^n \fint_{B_{\eps_n}(x_i)}\phi\dx- \int_Q\phi(x)\,\dx\right] &\leq \frac{3\,C_X}{n^{1/2}}\\
\sup_{\phi \in B^Y} \left[\frac1n\sum_{i=1}^n \fint_{B_{\eps_n}(x_i)}\phi\dx - \int_Q\phi(x)\,\dx\right] &\geq c_d\,n^{-1/d}\\
\sup_{\phi \in B^Z} \left[\frac1n\sum_{i=1}^n \fint_{B_{\eps_n}(x_i)}\phi\dx - \int_Q\phi(x)\,\dx\right] &\leq C_d
\end{align*}
hold. $c_d, C_d$ are dimension dependent constants and 
\[
\eps_n  = \gamma_d\,n^{-1/d}
\]
for a dimension-dependent $\gamma_d>0$.
\end{lemma}

\begin{proof}[Proof of Lemma \ref{lemma L2 approximation}]
{\bf $L^2$-estimate.} 
In all of the following, we rely on the interpretation of balls as periodic to avoid boundary effects. For a sample $S = (X_1,\dots,X_n)$ denote
\[
A_S(\phi) = \frac1n\sum_{i=1}^n \fint_{B_{\eps_n}(X_i)}\phi\dx
\]
Observe that
\begin{align*}
\sup_{\|\phi\|_{L^2}\leq 1}(A_{S}(\phi)-A(\phi)) &= \sup_{\|\phi\|_{L^2}\leq 1}\left[\frac1n\sum_{i=1}^n \fint_{B_{\eps_n}(X_i)}\phi\dx -\int\phi\dx\right]\\
	&= \sup_{\|\phi\|_{L^2}\leq 1, \int\phi=0}\frac1n\sum_{i=1}^n \fint_{B_{\eps_n}(X_i)}\phi\dx\\
	&= \sup_{\|\phi\|_{L^2}\leq 1, \int\phi=0} \frac1{n\,|B_{\eps_n}|}\int \left( \sum_{i=1}^n1_{B_{\eps_n}(X_i)}\right)\phi\dx\\
	&\leq \frac1{n\,\omega_d\,\eps_n^d} \left\| \sum_{i=1}^n1_{B_{\eps_n}(X_i)}\right\|_{L^2}.
\end{align*}
We compute
\begin{align*}
\left\| \sum_{i=1}^n1_{B_{\eps_n}(x_i)}\right\|_{L^2}^2&= \sum_{i=1}^n \left\|1_{B_{\eps_n}(x_i)}\right\|_{L^2}^2 + \sum_{i\neq j}\int 1_{B_{\eps_n}(x_i)}1_{B_{\eps_n}(x_j)}\dx\\
	&= n\,\omega_d|\eps_n|^d+ \sum_{i\neq j} \big|B_{\eps_n}(x_i) \cap B_{\eps_n}(x_j)\big|\showlabel\label{eq norm squared expression}
\end{align*}
It is easy to see that
\begin{align*}
\E_{(X_i, X_j) \sim U_{Q\times Q}} \big|B_{\eps_n}(X_i) \cap B_{\eps_n}(X_j)\big| &= \E_{X\sim U_Q} \big|B_{\eps_n}(X) \cap B_{\eps_n}(0)\big|\\
	&= \int_{B_{2\eps_n}} \big|B_{\eps_n}(x) \cap B_{\eps_n}(0)\big|\dx\\
	&= \eps_n^d \int_{B_2}\big|B_{\eps_n}\big(\eps_n\,x\big) \cap B_{\eps_n}(0)\big|\dx\\
	&= \eps_n^d \int_{B_2} \eps_n^d\,\big|B_1(x)\cap B_1(0)\big|\dx\\
	&= \eps_n^{2d}\,\omega_d2^d\,\frac1{\omega_d2^d} \int_{B_2} \big|B_1(x)\cap B_1(0)\big|\dx\\
	&= \eps_n^{2d}\,\bar c_d\,2^{d}\omega_d.\showlabel\label{eq intersection of balls}
\end{align*}
where
\[
\bar c_d := \frac1{\omega_d2^d} \int_{B_2} \big|B_1(x)\cap B_1(0)\big|\dx
\]
is a dimension-dependent constant.
Thus combining \eqref{eq norm squared expression} and \eqref{eq intersection of balls} we find that
\begin{align*}
\E_{S\sim (\L^d|_Q)^n}\left\| \sum_{i=1}^n1_{B_{\eps_n}(X_i)}\right\|_{L^2}^2 &= \omega_d\,\eps_n^d\left[ n + n(n-1)\,\bar c_d\,(2\eps_n)^d\right]\\
	&\leq \omega_d\,n\eps_n^d\left[ 1+ \bar c_d2^d \,n\eps_n^d\right].
\end{align*}
This allows us to estimate
\begin{align*}
\E_S\bigg[\sup_{\|\phi\|_{L^2}\leq 1}\big(A_S\phi - A\phi\big)\bigg] &\leq \frac1{n\,\omega_d\,\eps_n^d} \E_S\left\| \sum_{i=1}^n1_{B_{\eps_n}(X_i)}\right\|_{L^2}\\
	&\leq \frac1{n\,\omega_d\,\eps_n^d} \left(\E_S\left\| \sum_{i=1}^n1_{B_{\eps_n}(X_i)}\right\|_{L^2}^2\right)^\frac12\\
	&\leq \frac1{n\,\omega_d\,\eps_n^d} \sqrt{\omega_d\,n\eps_n^d\big[1+ c_d2^d\,n\eps_n^d\big]}\\
	&= \sqrt{\frac{1+ \bar c_d2^d\,n\eps_n^d}{\omega_d\,n\eps_n^d}}\\
	&= \sqrt{\frac{1+ \bar c_d2^d\,\gamma_d^d}{\omega_d\,\gamma_d^d}}
\end{align*}
when we choose
\[
\eps_n = \gamma_d\,n^{-1/d}
\]
for a dimension-dependent constant $\gamma_d$.

{\bf Lipschitz estimate.} If $E\subset \R^d$ is open and bounded, denote by $U_E$ the uniform distribution on $E$. Note that
\[
W_1\big(\delta_0, U_{B_\eps}\big) = \fint_{B_\eps} |x| \dx = \fint_{B_1} \eps|x|\dx = \eps \,\fint_{B_1} |x|\dx.
\]
Since the set of all transport plans between two measures given as convex combinations of measures is larger than the set of all plans which transport one term of the combination to another, we find that 
\begin{align*}
W_1\left(\frac1n \sum_{i=1}^n\delta_{x_i}, \frac1n\sum_{i=1}^n U_{B_{\eps_n}(x_i)}\right) &= \inf\left\{ \int |x-y|\,\d\pi_{x,y} \:\bigg|\: \pi \in \mathcal P \left(\frac1n \sum_{i=1}^n\delta_{x_i},\frac1n\sum_{i=1}^n U_{B_{\eps_n}(x_i)}\right) \right\} \\
	&\leq \inf\left\{ \int |x-y|\,\d\pi_{x,y} \:\bigg|\: \pi = \frac1n \sum_{i=1}^n\pi_i, \:\:\pi_i\in \mathcal P \left(\delta_{x_i}, U_{B_{\eps_n}(x_i)}\right) \right\} \\
	&= \frac1n\sum_{i=1}^n W_{1}(\delta_{x_i},U_{B_{\eps_n}(x_i)})\\
	&= \eps_n \,\fint_{B_1} |x|\dx\\
	&= \gamma_d\,n^{-1/d}\fint_{B_1} |x|\dx.
\end{align*}
We find by the triangle inequality that
\begin{align*}
W_1\left(\L^d|_Q, \frac1n\sum_{i=1}^n U_{B_{\eps_n}(x_i)}\right) &\geq W_1\left(\L^d|_Q, \frac1n\sum_{i=1}^n \delta_{x_i}\right)  - W_1\left(\frac1n \sum_{i=1}^n\delta_{x_i}, \frac1n\sum_{i=1}^n U_{B_{\eps_n}(x_i)}\right)\\
	&\geq \left[\frac d{d+1}\,\frac1{\big[(d+1)\,\omega_d\big]^{\frac1d}} - \gamma_d\fint_{B_1} |x|\dx\right]\,n^{-1/d}
\end{align*}
When we choose $\gamma_d$ small enough, we conclude as before that 
\[
\sup_{\phi \text{ is 1-Lipschitz}} \left(\frac1n\sum_{i=1}^n\fint_{B_{\eps_n(x_i)}}\phi\dx - \int_Q\phi\,\dx\right) \geq \,c_d\,n^{-1/d}
\]
for some positive $c_d>0$. Recall that this holds for {\em all} empirical measures, and that the Lipschitz constant is an equivalent norm for our purposes.

{\bf $X$-estimate.} We compute that
\begin{align*}
\E_{X_i \sim \L^d|_Q\text{ iid}} &\left\{\sup_{\phi \in B^X} \left[\frac1n\sum_{i=1}^n\int_{B_{\eps_n}(X_i)}\phi\dx - \int_Q\phi(x)\,\dx\right]\right\}\\
	&= \E_{X_i \sim \L^d|_Q\text{ iid}} \left\{\sup_{\phi \in B^X} \left[\frac1n\sum_{i=1}^n\int_{B_{\eps_n}(X_i)}\phi\dx - \int_Q\phi(x)\,\dx\right]\right\}\\
	&= \E_{X_i \sim \L^d|_Q\text{ iid}} \left\{\sup_{\phi \in B^X} \left[\frac1n\sum_{i=1}^n\int_{B_{\eps_n}(X_i)}\left(\phi(y) - \int_Q\phi(x)\,\dx \right)\dy\right]\right\}\\
	&=  \E_{X_i \sim \L^d|_Q\text{ iid}} \left\{\sup_{\phi \in B^X} \left[\int_{B_{\eps_n}}\frac1n\sum_{i=1}^n\left(\phi(X_i+ y) - \int_Q\phi(x)\,\dx \right)\dy\right]\right\}\\
	&\leq \E_{X_i \sim \L^d|_Q\text{ iid}} \left\{\int_{B_{\eps_n}} \sup_{\phi \in B^X} \left[\frac1n\sum_{i=1}^n\left(\phi(X_i+ y) - \int_Q\phi(x)\,\dx \right)\right]\dy\right\}\\
	&= \E_{X_i \sim \L^d|_Q\text{ iid}} \left\{ \sup_{\phi \in B^X} \left[\frac1n\sum_{i=1}^n\left(\phi(X_i) - \int_Q\phi(x)\,\dx \right)\right]\right\}\\
	&\leq \frac{C_X}{\sqrt n}
\end{align*}
since $X_i$ and $X_i + y$ have the same law (where $X_i+y$ is interpreted as a shift on the flat torus).

{\bf Conclusion.} Since the random variables 
\[
\sup_{\phi \in B^X} \left[\frac1n\sum_{i=1}^n\left(\phi(X_i) - \int_Q\phi(x)\,\dx \right)\right]
\]
are non-negative, we find that by Chebyshev's inequality that
\[
(\L^d|_Q)^n \left(\left\{(X_1,\dots,X_n)\:\bigg|\:\sup_{\phi \in B^X} \left[A_{X}(\phi)-A(\phi)\right]>\frac{3C_X}{\sqrt n}\right\}\right) \leq \frac13
\]
and similarly 
\[
(\L^d|_Q)^n \left(\left\{(X_1,\dots,X_n)\:\bigg|\:\sup_{\phi \in B^Z} \left[A_{X}(\phi)-A(\phi)\right]>3\gamma_d\,n^{-1/d}\fint_{B_1} |x|\dx\right\}\right) \leq \frac13.
\]
Since the $Y$-estimate is satisfied for any empirical measure, we conclude that there exists a set of points $x_1,\dots, x_n$ in $Q$ such that $A_n:= A_{(x_1,\dots,x_n)}$ satisfies the conditions of the theorem.
\end{proof}

\begin{corollary}\label{corollary l2 slow}
There exists a $1$-Lipschitz function $\phi$ on $Q$ such that 
\[
\limsup_{t\to\infty} \left(t^\gamma\,\inf_{\|f\|_X \leq t}\|\phi - f\|_{L^2(Q)}^2\right) = \infty.
\]
for all $\gamma> \frac{4}{d-2}$.
\end{corollary}

\cor{
\begin{example}
In \cite{deep_barron}, `tree-like' function spaces $\W^L$ for fully connected ReLU networks with $L$ hidden layers and bounded path-norm are introduced. There it is shown that the unit ball $B^L$ in $\W^L$ satisfies the  Rademacher complexity estimate
\[
\mathrm{Rad}(B^L, S) \leq 2^{L+1}\,\sqrt{\frac{2\,\log(d+2)}n}
\]
on any sample set of $n$ elements inside $[-1,1]^d$. According to \cite[Lemma 26.2]{shalev2014understanding}, we have
\[
\E_{(x_1,\dots,x_n)\sim\P^n} \left[\sup_{\|h\|_{\W^L}\leq 1} \left|\int_Qh(x) \,\P(\d x) - \frac1n\sum_{i=1}^nh(x_i)\right|\right]\leq 2\,\E_{(x_1,\dots,x_n)\sim\P^n} \, \mathrm{Rad}\big(B,\{x_1,\dots,x_n\}\big).
\]
for any data distribution $\P$ on $Q$, so $\W^L$ can be chosen as $X$ in this section. A more detailed consideration of Rademacher complexities in the case of Barron functions (for different activation functions) can be found in Appendix \ref{appendix barron} in the proof of Lemma \ref{lemma uniform monte-carlo}. For ReLU activation, Barron space coincides with the tree-like function space $\W^1$ with one hidden layer.
\end{example}
}

\cor{
\begin{corollary}
Let $X= \W^L(Q)$ the tree-like function space, $Y= C^{0,1}(Q)$ the space of Lipschitz-continuous functions (with respect to the $\ell^\infty$-norm on $\R^d$) and $Z= L^2(Q)$. Then
\[
\rho(t) = w_{d_Z}(B^X_t, B^Y_1) \geq 2^{-\frac1d} \frac{\big(c_d/2\big)^\frac{d}{d-2}}{C_d\,\left(3\cdot2^{L+2}\sqrt{\frac{2\,\log(2d+2)}{n}}\right)^\frac2{d-2}}\,t^{-\frac{2}{d-2}} = \bar c_d\,2^{-\frac{L}{d-2}}\,t^{-\frac2{d-2}}
\]
for a dimension-dependent constant $\bar c_d>0$.
\end{corollary}
In particular, slightly deeper neural networks do not possess drastically larger the approximation power compared in the class of Lipschitz functions. 
}


\section{Approximating Two-Layer Neural Networks by Kernel Methods}\label{section rkhs}

A brief review of reproducing kernel Hilbert spaces and our notation is given in Appendix \ref{appendix rkhs}. \cor{For bounded kernels}, the RKHS $\H_{k}$ embeds continuously into $L^2(\P)$. We assume additionally that $\H_{k}$ embeds into $L^2(\P)$ compactly. In Appendix \ref{appendix rkhs}, we show that this assumption is met for common random feature kernels and two-layer neural tangent kernels. Compactness allows us to apply the Courant-Hilbert Lemma, which is often used in the eigenvalue theory of elliptic operators.

\begin{lemma}\cite[Satz 8.39]{dobrowolski2010angewandte}
Let $H$ be a real, separable infinite-dimensional Hilbert space with two symmetric and continuous bilinear forms $B, K: H\times H\to \R$. Assume that
\begin{enumerate}
\item $K$ is continuous in the weak topology on $H$,
\item $K(u, u)>0$ for all $u\neq 0$, and
\item $B$ is coercive relative $K$, i.e.
\[
B(u,u) \geq c\,\|u\|_H^2 - \tilde c\,K(u,u)
\]
for constants $c, \tilde c>0$.
\end{enumerate}
Then the eigenvalue problem 
\[
B(u, u_i) = \lambda_i\,K(u,u_i)\quad\forall\ u\in H
\]
has countably many solutions $(\lambda_i, u_i)$. Every eigenvalue has finite multiplicity, and if sorted in ascending order then
\[
\lim_{i\to\infty}\lambda_i = +\infty.
\]
The space spanned by the eigenvectors $u_i$ is dense in $H$ and the eigenvectors satisfy the orthogonality relation
\[
K(u_i,u_j) = \delta_{ij}, \qquad  B(u_i,u_j) = \lambda_i\,K(u_i,u_j) = \lambda_i\delta_{ij}.
\]
We can expand the bilinear forms as
\[
B(u, v) = \sum_{i=1}^\infty \lambda_i\,K(u_i, u)\,K(u_i, v), \qquad K(u, v) =  \sum_{i=1}^\infty K(u_i, u)\,K(u_i, v).
\]
The pairs $(\lambda_i, u_i)$ are defined as solutions to the sequence of variational problems
\[
\lambda_i = B(u_i, u_i) = \inf\left\{\frac{B(u, u)}{K(u,u)}\:\bigg|\: K(u, u_j) = 0 \quad\forall\ 1\leq j\leq i-1\right\}.
\]
\end{lemma}

Consider $H= L^2(\P)$ and
\[
B(u,v) = \langle u,v\rangle_{L^2(\P)}, \qquad K(u,v) = \int_{\R^d\times\R^d}u(x)\,v(x')\,k(x,x')\,\P(\d x) \,\P(\d x') = \langle \overline Ku, v \rangle _{L^2(\P)}
\]
where 
\[
\ol Ku(x) = \int_{\R^d}k(x,x')\,u(x')\,\P(\d x').
\]
It is easy to see that the assumptions of the Courant-Hilbert Lemma are indeed satisfied. In particular, note that
\[
K(u, u_i) = \mu_i\,B(u,u_i)\quad\forall\ u \in L^2(\P) \qquad\LRa\qquad \ol Ku_i = \mu_i \,u_i.
\]
By definition, all eigenfunctions lie in the reproducing kernel Hilbert space.

Let $X = \H_{k}$ be a suitable RKHS, $Y = \mathcal B(\P)$ Barron space (see Appendix \ref{appendix barron} for a brief review), $Z = W = L^2(\P)$. Then $X, Y \embeds Z$. Furthermore, consider the sequence of $n$-dimensional spaces
\[
X_n = \mathrm{span}\{u_1,\dots, u_n\}\subseteq X
\]
spanned by the first $n$ eigenfunctions of $\ol K$ and the maps
\[
A_n: Z\to W, \qquad A_n = P_{X_n}, \qquad A = \id_{L^2(\P)}
\]
where $P_{V}$ denotes the $Z$-orthogonal projection onto the subspace $V$. Due to the orthogonality statement in the Courant-Hilbert Lemma, $P_V$ is also the $X$-orthogonal projection for this specific sequence of spaces.

\begin{lemma}\label{lemma slowness estimates}
If $\P = \L^d|_{[0,1]^d}$, we have the estimates
\[
\|A_n - A\|_{L(X,W)} \leq \frac1{\lambda_{n+1}^{1/2}}, \qquad \|A_n - A\|_{L(Y,W)} \geq \frac{c}d\,n^{-1/d}, \qquad \|A_n-A\|_{L(Z,W)} \leq 1
\]
where $c>0$ is a universal constant.
\end{lemma}

\begin{proof}[Proof of Lemma \ref{lemma slowness estimates}]
{\bf $X$-estimate.} Normalize the eigenfunctions $f_i$ of $k$ to be $L^2(\P)$-orthonormal. For $f\in \H_{k_\pi}$ we have the expansion $f = \sum_{i=1}^\infty a_i\,f_i$ and thus $(A-A_n)f = \sum_{i=n+1}^\infty a_i\,f_i$ such that
\begin{align*}
\|(A_n-A)f\|_W^2 = \sum_{i=n+1}^\infty|a_i|^2
	\leq \sum_{i=n+1}^\infty \frac{\lambda_i}{\lambda_{n+1}}\,|a_i|^2 = \frac{\|A_nf\|_X^2}{\lambda_{n+1}}
	\leq \frac{\|f\|_X}{\lambda_{n+1}} \leq \frac1{\lambda_{n+1}}
\end{align*}

{\bf $Y$-estimate.} See \cite[Theorem 6]{barron1993universal}. There it is shown that {\em any} sequence of $n$-dimensional spaces suffers from the curse of dimensionality when approximating a subset of Barron space. 

{\bf $Z$-estimate.} The orthogonal projection $A - A_n = P_{X_n^\bot}$ has norm one.
\end{proof}

Thus, if an RKHS has rapidly decreasing eigenvalues independently of the dimension of the ambient space (which is favourable from the perspective of statistical learning theory), then it suffers from a slow approximation property.  

\begin{example}\label{example rkhs slow}
In Appendix \ref{appendix rkhs}, we give examples of random feature kernels and neural tangent kernels for which $\lambda_k \leq c_d\,k^{-\frac12 + \frac3{\cor{2}d}}$. \cor{In Appendix \ref{appendix barron}, we briefly discuss Barron space.} Applying Lemma \ref{lemma slow sequence} with 
\[
\alpha = \frac14 - \frac3{\cor{4}d}, \qquad \beta = \frac1d\qquad \Ra\quad \frac{\beta}{\alpha-\beta} = \frac{\frac1d}{\frac{d}{4d} - \frac{\cor{3}}{4d} - \frac{4}{4d} } = \frac{4}{d-\cor{7}},
\]
we see that the $L^2$-width of the RKHS in Barron space is bounded from below by
\[
\rho(t) := \sup_{\|\phi\|_{\B(\P)} \leq 1} \inf_{\|\psi\|_{\H_{k,\P}}\leq t} \|\phi - \psi\|_{L^2(\P)}\:\geq c_d t^{-\frac{\beta}{\alpha-\beta}} = c_d\,t^{- \frac{4}{d-10}}.
\]
Due to Lemma \ref{lemma slow point}, there exists a function $\phi$ in Barron space such that
\[
\limsup_{t\to \infty} \left(t^\gamma\cdot \inf_{\|\psi\|_{\H_{k,\P}}\leq t} \|\phi - \psi\|_{L^2(\P)}\right) = \infty
\]
for all $\gamma> \frac{4}{d-\cor{7}}$.
\end{example}

\section{Discussion}\label{section discussion}

\cor{From the viewpoint of functional analysis and more precisely function spaces, a fundamental task in machine learning is balancing the approximation and estimation errors of a hypothesis class. In overly expressive function classes, it may be difficult to assess the performance of a function from a small data sample, whereas too restrictive function classes lack the expressivity to perform well in many problems. In this article, we made a first step in trying to quantify the competition between estimation and approximation. Our results show that linear function classes in which the estimation error is strongly controlled (including, but not limited to those developed for infinitely wide neural networks), the approximation error must suffer from the curse of dimensionality within the class of Lipschitz-functions. Additionally}, we emphasize that kernel methods (including some neural tangent kernels) are subject to the curse of dimensionality where adaptive methods like shallow neural networks are not. 

In a companion article \cite{dynamic_cod}, we show that the Barron norm and the RKHS norm increase at most linearly in time during gradient flow training in the mean field scaling regime. This means that the $L^2$-population risk of a shallow neural network or kernel function can only decay like $t^{-\alpha_d}$ for general Lipschitz or Barron target functions respectively, where $\alpha_d$ is close to zero in high dimension. It is therefore of crucial importance to understand the function spaces associated with neural network architectures under the natural path norms.

While the theory of function spaces for low-dimensional analysis (Sobolev, Besov, BV, BD, etc.) is well studied, the spaces for high-dimensional (but not infinite-dimensional) analysis is at its very beginning. To the best of our knowledge, the currently available models for neural networks consider infinitely wide two-layer \cor{or multi-layer} networks or infinitely deep networks with bounded width \cite{E:2019aa, deep_barron, weinan2019lei}. A different perspective on the approximation spaces of deep networks focussing on the number of parameters (but not their size) is developed in \cite{gribonval2019approximation}.

Even for existing function spaces, it is hard to check whether a given function belongs to the space. Barron's original work \cite{barron1993universal} shows (in modern terms) that every sufficiently smooth function on an extension domain belongs to Barron space, where the required degree of smoothness depends on the dimension. More precisely, if $f$ is a function on $\R^d$ such that its Fourier transform $\hat f$ satisfies
\[
C_f:= \int_{\R^d} \big|\hat f(\xi)\big|\,|\xi|\,\d\xi < \infty,
\]
then for every compact set $K\subset\R^d$ there exists a Barron function $g$ such that $f\equiv g$ on $\R^d$. In particular, if \cor{$f\in H^s(\R^d)$ where $s>\frac{d}2+1$, then $C_f<\infty$. In particular, if $f\in C^k(\Omega)$ for $k>\frac d2+1$ and $\Omega$ has smooth boundary, then by standard extension results we see that $f\in \B$.}
This holds for networks with {\em any} sigmoidal activation function or ReLU activation. The criterion is unsatisfying in two ways:

\begin{enumerate}
\item The function $f$ has to be defined on the whole space for Fourier-analytic considerations to apply. Given a function defined on a compact set, one has to find a good extension to the whole space.

\item The constant $C_f$ merely gives an upper bound on the Barron-norm of a two layer network. If $f\notin C^1$, then $C_f=+\infty$. If $f(x) = \sum_{i=1}^m a_i\,\mathrm{ReLU}(w_i^Tx+ b_i)$ is a two-layer neural network with finitely many nodes, then $f$ is only Lipschitz continuous and not $C^1$-smooth (unless it is linear). Thus the criterion misses many functions of practical importance.
\end{enumerate}

\subsection{Open Problems} 

Many questions in this field remain open.

\begin{enumerate}
\item The slow approximation property in $L^\infty$ is based purely on the slow convergence of empirical measures, while the $L^2$-construction also uses the translation-invariance of Lebesgue measure for convenience. Does a `curse of dimensionality' type phenomenon affect $L^2$-approximation when $\mathbb P$ has a density with respect to Lebesgue measure, or more generally is a regular measure concentrated on or close to a high-dimensional manifold in an even higher-dimensional ambient space?

\item We used Lipschitz functions for convenience, but we believe that a similar phenomenon holds for $C^k$ functions for any fixed $k$ which does not scale with dimension. To apply the same approach, we need to answer how quickly the $1$-Wasserstein-type distances
\[
\widetilde W_{C^k, \lambda}(\mu, \nu) = \sup\left\{ \int f\,\mu(\d x) - \int f\,\nu(\d x)\:\bigg|\: f\text{ is 1-Lipschitz, }|D^kf|_\infty \leq \lambda\right\}
\]
decay in expectation when $\nu= \mu_n$ is an empirical measure sampled iid from $\mu$. Other concepts of Wasserstein-type distance would lead to similar results.

\item To prove the curse of dimensionality phenomenon, we used a multi-scale construction with modifications on quickly diverging scales. The statement about the upper limit is only a `worst case curse of dimensionality' and describes slow convergence on an infinite set of vastly different scales. Replacing the upper limit in Theorem \ref{lemma slow point} by a lower limit would lead to a much stronger statement with more severe implications for applications.

\item Our proof of slow approximation was purely functional analytic and abstract. Is it possible to give a concrete example of a Lipschitz function which is poorly approximated by Barron functions of low norm in $L^2(\mathbb P)$ for a suitable data measure $\mathbb P$? 

More generally, is it possible to find general criteria to establish how well a given Lipschitz function (or even a given collection of data) can be approximated by a certain network architecture?

\item We proved that a Lipschitz function $\phi$ exists for which 
\[
\inf_{\|f\|_{\text{Barron}}\leq t} \|f-\phi\|_{L^2} \geq c_d\,t^{-1/d}
\]
on a suitable collection of scales $t_k$. Can we more generally characterize the class
\[
Y_\alpha := \left\{y\in Y\:\bigg|\: \limsup_{t\to\infty}\big[ t^\alpha \dist_{Z}(y, t\cdot B^X) \big] < \infty\right\}
\]
for $\alpha>0$, $X = $ Barron space, $Y=$ Lipschitz space and $Z = L^2$? This question arises naturally when considering training algorithms which increase the complexity-controlling norm only slowly. It is related, but not identical to considerations in the theory of real interpolation spaces.

\item The draw-back of the functional analytic approach of this article is that it does not encompass many common loss functionals such as logistic loss. Is it possible to show by different means that they are subject to similar problems in high dimension?

\item We give relevant examples of reproducing kernel Hilbert spaces in which the eigenvalues of the kernel decay at a dimension-independent rate (at leading order). To the best of our knowledge, a general perspective on the decay of eigenvalues of kernels in practical applications without strong symmetry assumptions is still missing.
\end{enumerate}

\appendix
\section{A Brief Review of Barron Space}\label{appendix barron}

For the convenience of the reader, we recall Barron space for two-layer neural networks as introduced by E, Ma and Wu \cite{weinan2019lei,E:2018ab}. We focus on the functional analytic properties of Barron space, for results with a focus on machine learning we refer the reader to the original sources. The same space is denoted as $\F_1$ in \cite{bach2017breaking}, but described from a different perspective. For functional analytic notions, we refer the reader to \cite{MR2759829}.

Let $\mathbb{P}$ be a probability measure on $\R^d$ and $\sigma$ a Lipschitz-continuous function such that either

\begin{enumerate}
\item $\sigma = \mathrm{ReLU}$ or
\item $\sigma$ is sigmoidal, i.e.\ $\lim_{z\to \pm \infty} \sigma(z) = \pm 1$ (or $0$ and $1$).
\end{enumerate}

Consider the class $\F_m$ of two-layer networks with $m$ neurons
\[
f_\Theta (x) = \frac1m \sum_{i=1}^m a_i \,\sigma\big(w_i^Tx+ b_i\big), \qquad\Theta = \{(a_i, w_i, b_i) \in \R^{d+2}\}_{i=1}^m.
\]
It is well-known that the closure of $\F = \bigcup_{m=1}^\infty\F_m$ in the uniform topology is the space of continuous functions, see e.g.\ \cite{cybenko1989approximation}. Barron space is a different closure of the same function class where the path-norm
\[
\|f_\Theta\|_{\text{path}} = \frac1m \sum_{i=1}^m |a_i| \,\big[|w_i|_{\ell^q} + |b_i|\big]
\] 
remains bounded. Here we assume that data space $\R^d$ is equipped with the $\ell^p$-norm and take the dual $\ell^q$-norm on $w$. The concept of path norm corresponds to ReLU activation, a slightly different path norm for bounded Lipschitz activation is discussed below. 

The same class is often discussed without the normalizing factor of $\frac1m$. With the factor, the following concept of infinitely wide two-layer networks emerges more naturally.

\begin{definition}
Let $\pi$ be a Radon probability measure on $\R^{d+2}$ with finite second moments, \cor{which we denote by $\pi \in \mathcal P_2(\R^{d+2})$}. We denote by
\[
f_\pi(x) = \int_{\R^{d+2}} a\,\sigma(w^Tx+b)\,\pi(\d a \otimes \d w\otimes \d b)
\]
the two-layer network associated to $\pi$.
\end{definition}

If $\sigma = \mathrm{ReLU}$, it is clear that $f_\pi$ is Lipschitz-continuous on $\R^d$ with Lipschitz-constant $\leq \|f_\pi\|_{\B(\P)}$, so $f_\pi$ lies in the space of (possibly unbounded) Lipschitz functions $\mathrm{Lip}(\R^d)$. If $\sigma$ is a bounded Lipschitz function, the integral converges in $C^0(\R^d)$ without assumptions on the moments of $\pi$. If the second moments of $\pi$ are bounded, $f_\pi$ is a Lipschitz function also in the case of bounded Lipschitz activation $\sigma$. 

For technical reasons, we will extend the definition to distributions $\pi$ for which only the {\em mixed second moments}
\[
\int_{\R^{d+2}}|a|\,\big[|w| + \psi(b)\big]\,\pi(\d a\otimes \d w \otimes \d b)
\]
are finite, where $\psi = |b|$ in the ReLU case and $\psi = 1$ otherwise. Now we introduce the associated function space.

\begin{definition}
We denote
\[
\|\cdot\|_{\B(\P)}:\mathrm{Lip}(\R^d)\to[0,\infty), \qquad \|f\|_{\B(\P)} = \inf_{\{\pi|f_\pi =f\:\P-\text{a.e.}\}} \int_{\R^{d+2}}|a|\,\big[|w|_{\ell^q}+|b|\big]\,\pi(\d a\otimes \d w \otimes \d b)
\]
if $\sigma = \mathrm{ReLU}$ and
\[
\|\cdot\|_{\B(\P)}:\mathrm{Lip}(\R^d)\to[0,\infty), \qquad \|f\|_{\B(\P)} = \inf_{\{\pi|f_\pi =f\:\P-\text{a.e.}\}} \int_{\R^{d+2}}|a|\,\big[|w|_{\ell^q}+1\big]\,\pi(\d a\otimes \d w \otimes \d b)
\]
otherwise. In either case, we denote 
\[
\B(\P) = \big\{f\in \mathrm{Lip}(\R^d) : \|f\|_{\B(\P)}<\infty\big\}.
\]
Here $\inf\emptyset = + \infty$.
\end{definition}

\begin{remark}\label{remark no optimal measure}
It depends on the activation function $\sigma$ whether or not the infimum in the definition of the norm is attained. If $\sigma =$ ReLU, this can be shown to be true by using homogeneity. Instead of a probability measure $\pi$ on the whole space, one can use a signed measure $\mu$ on the unit sphere to express a two layer network. The compactness theorem for Radon measures provides the existence of a measure minimizer, which can then be lifted to a probability measure (see below for similar arguments). On the other hand, if $\sigma$ is a classical sigmoidal function such that 
\[
\lim_{z\to -\infty} \sigma(z) = 0, \qquad \lim_{x\to \infty}\sigma(z) = 1, \qquad 0 < \sigma(z) < 1 \quad \forall\ z\in \R,
\]
then the function $f(z) \equiv 1$ has Barron norm $1$, but the infimum is not attained. This holds true for {\em any} data distribution $\P$.

{\em Proof.} 
\begin{enumerate}
\item For any $x\in \spt(\P)$ we have
\[
1 = f(x) = \int_{\R^{d+2}} a\,\sigma(w^Tx+b)\,\d\pi \leq \int_{\R^{d+2}}|a|\,\d\pi \leq \int_{\R^{d+2}}|a|\,\big[|w| + 1\big]\,\d\pi.
\]
Taking the infimum over all $\pi$, we find that $1\leq \|f\|_{\B(\P)}$. For any measure $\pi$, the inequality above is strict since $|\sigma|<1$, so there is no $\pi$ which attains equality.
\item We consider a family of measures 
\[
\pi_\lambda = \delta_{a= \sigma(\lambda)^{-1}}\,\delta_{w=0}\,\delta_{b = \lambda} \qquad \Ra\quad f_{\pi_\lambda} \equiv 1\quad\text{and } \int_{\R^{d+2}} |a|\big[|w|_{\ell^q}+1\big]\,\d\pi_\lambda = \frac1{\sigma(\lambda)} \to 1
\]
as $\lambda\to \infty$.
\end{enumerate}
Thus $\|f\|_{\B(\P)} = 1$, but there is no minimizing parameter distribution $\pi$.
\end{remark}

\begin{remark}
The space $\B(\P)$ does not depend on the measure $\P$, but only on the system of null sets for $\P$.
\end{remark}

We note that the space $\B(\P)$ is reasonably well-behaved from the point of view of functional analysis.

\begin{lemma}
$\B(\mathbb P)$ is a Banach space with norm $\|\cdot\|_{\B(\P)}$. If $\spt(\P)$ is compact, $\B(\P)$ embeds continuously into the space of Lipschitz functions $C^{0,1}(\spt\,\P)$.
\end{lemma}

\begin{proof}
{\bf Scalar multiplication.}
Let $f\in \B(\P)$. For $\lambda\in \R$ and $\pi \in \mathcal P_2(\R^{d+2})$, define the push-forward
\[
T_{\lambda\sharp}\pi \in \cor{\mathcal P}_2(\R^{d+2})\qquad\text{along } T_\lambda:\R^{d+2}\to\R^{d+2}, \quad T_\lambda (a,w,b) = (\lambda a, w,b).
\]
Then 
\[
f_{T_{\lambda\sharp}\pi} = \lambda \,f_\pi, \qquad \int_{\R^{d+2}} \,|a|\,\big[|w|+ 1\big]\,T_{\lambda\sharp}\pi(\d a \otimes \d w \otimes \d b) = |\lambda| \int_{\R^{d+2}} \,|a|\,\big[|w|+ 1\big]\,\pi(\d a \otimes \d w \otimes \d b)
\]
and similarly in the ReLU case. Thus scalar multiplication is well-defined in $\B(\mathbb P)$. Taking the infimum over $\pi$, we find that $\|\lambda f\|_{\B(\P)} = |\lambda|\,\|f\|_{\B(\P)}$.

{\bf Vector addition.}
Let $g,h \in \B(\P)$. Choose $\pi_g, \pi_h$ such that $g= f_{\pi_g}$ and $h = f_{\pi_h}$. Consider
\[
\pi = \frac12\big[ T_{2\sharp} \pi_g + T_{2\sharp} \pi_h\big]
\]
like above. Then $f_\pi = g+h$ and
\[
\int_{\R^{d+2}} \,|a|\,\big[|w|+ 1\big]\,\pi(\d a \otimes \d w \otimes \d b) = \int_{\R^{d+2}} \,|a|\,\big[|w|+ 1\big]\,\pi_g(\d a \otimes \d w \otimes \d b) + \int_{\R^{d+2}} \,|a|\,\big[|w|+ 1\big]\,\pi_h(\d a \otimes \d w \otimes \d b).
\]
Taking infima, we see that $\|g+h\|_{\B(\P)} \leq \|g\|_{\B(\P)} + \|h\|_{\B(\P)}$. The same holds in the ReLU case.

{\bf Positivity and embedding.}
Recall that the norm on the space of Lipschitz functions on a compact set $K$ is
\[
\|f\|_{C^{0,1}(K)} = \sup_{x\in K} |f(x)| + \sup_{x,y\in K,\: x\neq y} \frac{|f(x) - f(y)|}{|x-y|}.
\]
It is clear that $\|\cdot\|_{\B(\P)} \geq 0$. If $\sigma$ is a bounded Lipschitz function, then 
\[
\|f_\pi\|_{L^\infty(\spt\,\P)} \leq {\|f_\pi\|_{\B(\P)}}\,\cor{\|\sigma\|_{L^\infty(\R)}}
\]
like in Remark \ref{remark no optimal measure}. If $\sigma = \mathrm{ReLU}$, then 
\[
\sup_{x\in \R^d} \frac{|f_\pi(x)|}{1+ |x|} \leq \|f_\pi\|_{\B(\P)}.
\]
In either case
\[
|f_\pi(x) - f_\pi(y)| \leq [\sigma]_{\mathrm{Lip}}\,\|f_\pi\|_{\B(\P)}
\]
for all $x,y$ in $\spt(\P)$. In particular, $\|f\|_{\B(\P)} > 0$ whenever $f\neq 0$ in $\B(\P)$ and if $\spt(\P)$ is compact, $\B(\P)$ embeds into the space of Lipschitz functions on $\B(\P)$. If $\sigma$ is a bounded Lipschitz function, $\B(\P)$ embeds into the space of bounded Lipschitz functions also on unbounded sets.

{\bf Completeness.} Completeness is proved most easily by introducing a different representation for Barron functions. Consider the space
\[
V = \left\{\mu\:\bigg|\: \mu\text{ (signed) Radon measure on }\R^{d+2}\text{ s.t. }\int_{\R^{d+2}}|a|\,\big[|w|+|b|\big]\,|\mu|(\d a \otimes \d w\otimes \d b) < \infty\right\}
\]
where $|\mu|$ is the total variation measure of $\mu$. Equipped with the norm
\[
\|\mu\|_V = \int_{\R^{d+2}}|a|\,\big[|w|+|b|\big]\,|\mu|(\d a \otimes \d w\otimes \d b),
\]
$V$ is a Banach space when we quotient out measures supported on $\{|a|=0\}\cup \{|b| = |w| = 0\}$ or restrict ourselves to the subspace of measures such that $|\mu|(\{|a|=0\}) = |\mu| (\{|b|=|w|=0\}) = 0$. The only non-trivial question is whether $V$ is complete. By definition $\mu_n$ is a Cauchy sequence in $V$ if and only if $\nu_n := |a|\,[|w|+ |b|]\cdot\mu_n$ is a Cauchy sequence in the space of finite Radon measures. Since the space of finite Radon measures is complete, $\nu_n$ converges (strongly) to a measure $\nu$ which satisfies $\nu(\{a =0\}\cup\{(w,b) = 0\}) =0$. We then obtain $\mu:= |a|^{-1}\,[|w|+ |b|]^{-1}\cdot\nu$. For $\mu\in V$ we write
\[
f_\mu (x) = \int_{\R^{d+2}} a\,\sigma(w^Tx+b)\,\mu(\d a \otimes \d w \otimes \d b)
\]
and consider the subspace
\[
V^0_\P = \{\mu\in V\:|\: f_\mu = 0 \:\:\P-\text{almost everywhere}\}.
\]
Since the map
\[
V\mapsto C^{0,1}(\spt\,\P), \qquad \mu\mapsto f_\mu
\]
is continuous by the same argument as before, we find that $V^0_\P$ is a closed subspace of $V$. In particular, $V/V^0_\P$ is a Banach space. We claim that $\B(\P)$ is isometric to the quotient space $V/V^0_\P$ by the map $[\mu]\mapsto f_\mu$ where $\mu$ is any representative in the equivalence class $[\mu]$.

It is clear that any representative in the equivalence class induces the same function $f_\mu$ such that the map is well-defined. Consider the Hahn decomposition $\mu = \mu^+ - \mu^-$ of $\mu$ as the difference of non-negative Radon measures. Set
\[
m^\pm:= \|\mu^\pm\|, \qquad \pi = \frac12 \left[ \frac 1{m^+} T_{\cor{2}m^+\sharp}\mu^+ + \frac1{m^-}\,T_{-\cor{2}m^-\sharp}\mu^-\right].
\]
Then $\pi$ is a probability Radon measure such that $f_\pi = f_\mu$ and 
\[
\int_{\R^{d+2}}|a|\,\big[|w|+|b|\big]\,\pi(\d a \otimes \d w\otimes \d b)  = \int_{\R^{d+2}}|a|\,\big[|w|+|b|\big]\,|\mu|(\d a \otimes \d w\otimes \d b).
\]
In particular, $\|f_\mu\|_{\B(\P)} \leq \|\mu\|_V$. Taking the infimum of the right hand side, we conclude that $\|f_\mu\|_{\B(\P)} \leq \|[\mu]\|_{V/V^0_\P}$. The opposite inequality is trivial since every probability measure is in particular a signed Radon measure.

Thus $\B(\P)$ is isometric to a Banach space, hence a Banach space itself. We presented the argument in the context of ReLU activation, but the same proof holds for bounded Lipschitz activation.
\end{proof}

A few remarks are in order.

\begin{remark}\label{remark compact support}
The requirement that $\mathbb P$ have compact support can be relaxed when we consider the norm
\[
\|f\|_{C^{0,1}(\P)} = \sup_{x,y\in \spt\P,\: x\neq y} \frac{|f(x) - f(y)|}{|x-y|} + \|f\|_{L^1(\P)}
\]
on the space of Lipschitz functions. Since Lipschitz functions grow at most linearly, this is well-defined for all data distributions $\P$ with finite first moments.
\end{remark}

\begin{remark}
For general Lipschitz-activation $\sigma$ which is neither bounded nor ReLU, the Barron norm is defined as
\[
 \|f\|_{\B(\P)} = \inf_{\{\pi|f_\pi =f\:\P-\text{a.e.}\}} \int_{\R^{d+2}}|a|\,\big[|w|_{\ell^q}+|b| + 1\big]\,\pi(\d a\otimes \d w \otimes \d b).
\]
Similar results hold in this case.
\end{remark}

\begin{remark}
In general, $\B(\P)$ for ReLU activation is not separable. Consider $\P = \L^1|_{[0,1]}$ to be Lebesgue measure on the unit interval in one dimension. For $\alpha \in (0,1)$ set $f_\alpha(x) = \sigma(x-\alpha)$. Then for $\beta>\alpha$ we have
\[
1 =  \frac{\beta - \alpha}{\beta-\alpha} = \frac{(f_\alpha-f_\beta)(\beta) - (f_\alpha - f_\beta)(\alpha)}{\beta-\alpha} \leq [f_\beta-f_\alpha]_{\mathrm{Lip}} \leq \|f_\beta - f_\alpha\|_{\B(\P)}.
\]
Thus there exists an uncountable family of functions with distance $\geq 1$, meaning that $\B(\P)$ cannot be separable. 
\end{remark}

\begin{remark}
In general, $\B(\P)$ for ReLU activation is not reflexive. We consider $\P$ to be the uniform measure on $[0,1]$ and demonstrate that $\B(\P)$ is the space of functions whose first derivative is in $BV$ (i.e.\ whose second derivative is a Radon measure) on $[0,1]$. \cor{The space of Radon measures on $[0,1]$ is denoted by $\M[0,1]$ and equipped with the total variation norm.} 

Assume that $f$ is a Barron function on $[0,1]$. Then 
\begin{align*}
f(x) &= \int_{\R^3} a\,\sigma(wx + b)\,\pi(\d a \otimes \d w \otimes \d b)\\
	&= \int_{\{w\neq 0\}} a\,|w|\,\sigma\left(\frac{w}{|w|}x + \frac{b}{|w|}\right)\,\pi(\d a \otimes \d w \otimes \d b) + \int_{\{w=0\}}a\,\sigma(b)\pi(\d a \otimes \d w \otimes \d b)\\
	&= \int_{\R} \sigma\left(-x+\tilde b\right)\,\mu_1(\d \tilde b) + \int_{\R} \sigma\left(x + b\right)\,\mu_2(\d \tilde b) + \cor{\int_{\{w=0\}}a\,\sigma(b)\pi(\d a \otimes \d w \otimes \d b)}
\end{align*}
where 
\[
\mu_1 = T_\sharp \big(a\,|w|\cdot \pi\big), \qquad T:\{w<0\}\to \R, \quad T(a,w, b) = \frac{b}{|w|}.
\]
and similarly for $\mu_2$. The Barron norm is expressed as 
\[
\|f\|_{\B(\P)} = \inf_{\mu_1,\mu_2} \left[\int_{\R} 1 + |\tilde b|\,|\mu_1|(\d \tilde b) + \int_{\R} 1 + |\tilde b|\,|\mu_2|(\d \tilde b)\right] + \int_{\{w=0\}} |ab|\,\pi(\d a\otimes \d w \otimes \d b).
\]
Since $\sigma'' = \delta$, we can formally calculate that 
\[
f'' = \mu_1 + \mu_2
\]
This is easily made rigorous in the distributional sense. Since $[0,1]$ is bounded by $1$, we obtain in addition to the bounds on $f(0)$ and the Lipschitz constant of $f$ that
\[
\|f''\|_{\M[0,1]} \leq 2\,\|f\|_{\B(\P)}. 
\]
On the other hand, if $f$ has a second derivative, then
\begin{align*}
f(x) &= f(0) + f'(0)\,x + \int_0^x \int_0^t f''(\xi)\d\xi\,\d t\\
	&= f(0) + f'(0)\,x + \int_0^x (x-t)\,f''(t)\d t\\
	&= f(0)\,\sigma(1) + f'(0)\,\sigma(x) + \int_0^xf''(t)\,\sigma(x-t)\dt
\end{align*}
for all $x\in[0,1]$. This easily extends to measure valued derivatives and we conclude that
\[
\|f\|_{\B(\P)} \leq |f(0)| + |f'(0)| + \|f''\|_{\M[0,1]} .
\]
We can thus express $\B(\P) = BV([0,1]) \times \R\times\R$ with an equivalent norm
\[
\|f\|_{\B (\P)}' =  |f(0)| + |f'(0)| + \|f''\|_{\M[0,1]}.
\]
Thus $\B(\P)$ is not reflexive since $BV$ is not reflexive.
\end{remark}

Finally, we demonstrate that integration by empirical measures converges quickly on $\B(\P)$. Assume that $p=\infty$, $q=1$.

\begin{lemma}\label{lemma uniform monte-carlo}
The uniform Monte-Carlo estimate 
\[
\E_{X_i \sim \P} \left\{\sup_{\|\phi\|_{\B(\P)}\leq 1} \left[\frac1n\sum_{i=1}^n\phi(X_i) - \int_Q\phi(x)\,\P(\d x)\right]\right\} \leq  2L\,\sqrt{\frac{2\,\log(2d)}n}
\]
holds for any probability distribution $\P$ such that $\spt(\P)\subseteq[-1,1]^d$. Here $L$ is the Lipschitz-constant of $\sigma$.
\end{lemma}

\begin{proof}
A single data point may underestimate or overestimate the average integral, and the proof of convergence relies on these cancellations. A convenient tool to formalize cancellation and decouple this randomness from other effects is through Rademacher complexity \cite[Chapter 26]{shalev2014understanding}.

We denote $S= \{X_1,\dots, X_n\}$ and assume that $S$ is drawn iid from the distribution $\P$. We consider an auxiliary random vector $\xi$ such that the entries $\xi_i$ are iid (and independent of $S$) variables which take the values $\pm 1$ with probability $1/2$. Furthermore, abbreviate by $B$ the unit ball in $\B(\P)$. Furthermore
\begin{align*}
\mathrm{Rep}(B,S) &= \sup_{\phi \in B} \left[\frac1n\sum_{i=1}^n\phi(X_i) - \int_Q\phi(x)\,\P(\d x)\right]\\
\mathrm{Rad} (B,S)&= \E_\xi \sup_{\phi \in B} \left[\frac1n \sum_{i=1}^n \xi_i\,\phi(X_i)\right].
\end{align*}

According to \cite[Lemma 26.2]{shalev2014understanding}, the Rademacher complexity $\mathrm{Rad}$ bounds the representativeness of the set $S$ by
\[
\E_S \mathrm{Rep}(B,S) \leq 2\,\E_S \cor{\mathrm{Rad}}(B,S).
\]
The unit ball in Barron space is given by convex combinations of functions 
\[
\phi_{w,b} (x) = \pm \frac{\sigma(w^T x+ b)}{|w|+1} \quad\text{or} \quad \pm \frac{\sigma(w^T x+ b)}{|w|+|b|}
\]
respectively, so for fixed $\xi$, the linear map $\phi\mapsto  \frac1n \sum_{i=1}^n \xi_i\,\phi(X_i)$ at one of the functions in the convex hull, i.e.
\[
\mathrm{Rad} (B,S)= \E_\xi \sup_{(w,b)} \left[\frac1n \sum_{i=1}^n \xi_i\,\frac{\sigma(w^T X_i+ b)}{|w|+1}\right].
\]
According to the Contraction Lemma \cite[Lemma 26.9]{shalev2014understanding}, the Lipschitz-nonlinearity $\sigma$ can be neglected in the computation of the complexity. If $L$ is the Lipschitz-constant of $\sigma$, then
\begin{align*}
 \E_\xi \sup_{(w,b)} \left[\frac1n \sum_{i=1}^n \xi_i\,\frac{\sigma(w^T X_i+ b)}{|w|+1}\right] 
 	&\leq L \, \E_\xi \sup_{(w,b)} \left[\frac1n \sum_{i=1}^n \xi_i\,\frac{w^T X_i+ b}{|w|+1}\right]\\
	&= L\, \E_\xi \sup_{w\in \R^d} \frac{w^T}{|w|+1} \frac1n \sum_{i=1}^n \xi_iX_i\\
	&= L\, \E_\xi \sup_{|w|\leq 1} w^T \frac1n \sum_{i=1}^n \xi_iX_i\\
	&\leq L\,\sup_i|X_i|_\infty\,\sqrt{\frac{2\,\log(2d)}n}
\end{align*}
where we used \cite[Lemma 26.11]{shalev2014understanding} for the complexity bound of the linear function class and \cite[Lemma 26.6]{shalev2014understanding} to eliminate the scalar translation. A similar computation can be done in the ReLU case.
\end{proof}

\section{A Brief Review of Reproducing Kernel Hilbert Spaces, Random Feature Models, and the Neural Tangent Kernel}\label{appendix rkhs}

For an introduction to kernel methods in machine learning, see e.g.\ \cite[Chapter 16]{shalev2014understanding} or \cite{cho2009kernel} in the context of deep learning. Let $k:\R^d\times \R \to \R$ be a symmetric positive definite kernel. The {\em reproducing kernel Hilbert space} (RKHS) $\H_k$ associated with $k$ is the completion of the collection of functions of the form
\[
h(x) = \sum_{i=1}^n a_i\,k(x,x_i)
\]
under the scalar product $\langle k(x,x_i), k(x, x_j)\rangle_{\H_{k}} = k(x_i, x_j)$. Note that, due to the positive definiteness of the kernel, the representation of a function is unique.

\subsection{Random Feature Models}

The random feature models considered in this article are functions of the same form
\[
f(x) = \frac1m \sum_{i=1}^m a_i\,\sigma\big(w_i^Tx + b_i\big)
\]
as a two-layer neural network. Unlike neural networks, $(w_i, b_i)$ are not trainable variables and remain fixed after initialization. Random feature models are linear (so easier to optimize) but less expressive than shallow neural networks. While two-layer neural networks of infinite width are modelled as
\[
f(x) = \int_{\R^{d+2}} a\,\sigma(w^Tx+b)\,\pi(\d a\otimes \d w \otimes \d b)
\]
for variable probability measures $\pi$, an infinitely wide random feature model is given as
\[
f(x) = \int_{\R^{d+1}} a(w,b)\, \sigma(w^Tx+b)\,\pi^0(\d w\otimes \d b)
\]
for a fixed distribution $\pi^0$ on $\R^{d+1}$. One can think of Barron space as the union over all random feature spaces. It is well known that neural networks can represent the same function in different ways. For ReLU activation, there is a degree of degeneracy due to the identity
\[
0 = x+\alpha - (x+\alpha) = \sigma(x) - \sigma(-x) + \sigma(\alpha)- \sigma(-\alpha) - \sigma(x+\alpha) + \sigma(-x-\alpha). 
\]
This can be generalized to higher dimension and integrated in $\alpha$ to show that the random feature representation of functions where $\pi^0$ is the uniform distribution on the sphere has a similar degeneracy. For given $\pi^0$, denote
\[
N = \left\{a \in L^2(\cor{\pi^0})\:\bigg|\: \int_{\R^{d+1}} a(w,b)\, \sigma(w^Tx+b)\,\pi^0(\d w\otimes \d b) = 0 \quad\P-\text{a.e.}\right\}
\]

\begin{lemma}\cite[Proposition 4.1]{rahimi2008uniform}
The space of random feature models is dense in the RKHS for the kernel
\[
k(x,x') = \int_{\R^{d+1}} \sigma(w^Tx+b)\,\sigma(w^Tx' + b)\,\pi^0(\d w\otimes \d b)
\]
and if
\[
f(x) = \int_{\R^{d+1}} a(w,b) \sigma(w^Tx+b)\,\pi^0(\d w\otimes \d b)
\]
for $a\in L^2(\pi^0)/N$, then
\[
\|f\|_{\H_k} = \|a\|_{L^2(\pi^0)/N}.
\]
\end{lemma}

Note that the proof in the source uses a different normalization. The result in this form is achieved by setting $a= \alpha/p, b = \beta/p$ in the notation of \cite{rahimi2008uniform}.

\begin{corollary}
Any function in the random feature RKHS is Lipschitz-continuous and $[f]_{\mathrm{Lip}} \leq \|f\|_{\H_{k}}$. Thus $\H_{k}$ embeds compactly into $C^0(\spt\,\P)$ by the Arzel\`a-Ascoli theorem and a fortiori into $L^2(\P)$ for any compactly supported measure $\P$.
\end{corollary}

Let $\P$ be a compactly supported data distribution on $\R^d$. Then the kernel $k$ acts on $L^2(\P)$ by the map
\[
\ol K:L^2(\P)\to L^2(\P), \qquad \ol Ku(x) = \int_{\R^d} u(x')\,k(x,x')\,\P(\d x').
\]
Computing the eigenvalues of the kernel $k$ for a given parameter distribution $\pi^0$ and a given data distribution $\P$ is a non-trivial endeavor. The task simplifies considerably under the assumption of symmetry, but remains complicated. The following results are taken from \cite[Appendix D]{bach2017breaking}, where more general results are proved for $\alpha$-homogeneous activation for $\alpha\geq0$. We specify $\alpha=1$. 

\begin{lemma}
Assume that $\pi^0 = \P = \alpha_d^{-1}\cdot \H^d|_{S^d}$ where $S^d$ is the Euclidean unit sphere, $\alpha_d$ is its volume and $\sigma =$ ReLU. Then the eigenfunctions of the kernel $k$ are the spherical harmonics. The $k$-th eigenvalue $\lambda_k$ (counted without repetition) occurs with the same multiplicity $N(d,k)$ as eigenfunctions to the $k$-th eigenvalue of the Laplace-Beltrami operator on the sphere (the spherical harmonics). Precisely
\[
N(d,k) = \frac{2k+d-1}k \binom{k+d-2}{d-1}\qquad\text{ and }\qquad\lambda_k = \frac{d-1}{2\pi}\,2^{-k}\,\frac{\Gamma(d/2)\,\Gamma(k-1)}{\Gamma(k/2)\,\Gamma(\frac{k+d+2}2)}
\]
for $k\geq 2$.
\end{lemma}

We can extract a decay rate for the eigenvalues counted with repetition by estimating the height and width of the individual plateaus of eigenvalues. Denote by $\mu_i$ the eigenvalues of $k$ counted as often as they occur, i.e.\ 
\[
\mu_i = \lambda_k \quad\LRa\quad \sum_{j=1}^{k-1} N(d,j) < i \leq \sum_{j=1}^{k} N(d,j).
\]
By Stirling's formula one can estimate that for fixed $d$
\[
\lambda_k \sim k^{-\frac{d-3}2}
\]
as $k\to \infty$ where we write $a_k \sim b_k$ if and only if 
\[
0 < \liminf_{k\to \infty} \frac{a_k}{b_k} \leq \limsup_{k\to\infty} \frac{a_k}{b_k} < \infty.
\]
On the other hand
\[
N(d,k) = \frac dk \binom{k+d-1}{d} = \frac dk \frac{(k+d-1)\dots k}{d!} \sim \frac{(k+d)^{d-1}}{(d-1)!}.
\]
In particular, if $\mu_i = \lambda_k$, then
\begin{align*}
i &\sim C(d) \sum_{j=1}^k j^{d-1}
	\sim C(d) \int_1^{k+1} t^{d-1}\dt
	\sim C(d)\,k^d.
\end{align*}
Thus
\[
\mu_i = \lambda_k \sim k^{-\frac{d-3}2} \sim i^{-\frac{d-3}{2d}} = i^{-\frac12 + \frac{3}{2d}}.
\]

\begin{corollary}
Consider the random feature model for ReLU activation when both parameter and data measure are given by the uniform distribution on the unit sphere. Then the eigenvalues of the kernel decay like $i^{-\frac12 + \frac 3{2d}}$.
\end{corollary}

\begin{remark}
It is easy to see that up to a multiplicative constant, all radially symmetric parameter distributions $\pi^0$ lead to the same random feature kernel. This can be used to obtain an explicit formula in the case when $\pi^0$ is a standard Gaussian in $d$ dimensions. Since $k$ only depends on the angle between $x$ and $x'$, we may assume that $x=e_1$ and $x' = \cos \phi\,e_1 + \sin\phi \,e_2$ with $\phi \in[0,\pi]$. Now, one can use that the projection of the standard Gaussian onto the $w_1w_2$-plane is a lower-dimensional standard Gaussian. Thus the kernel does not depend on the dimension. An explicit computation in two dimensions shows that
\[
k(x,x') = \frac{\pi - \phi}\pi\,\cos \phi + \frac{\sin\phi}\pi,
\]
see \cite[Section 2.1]{cho2009kernel}.
\end{remark}

\subsection{The Neural Tangent Kernel}

The neural tangent kernel \cite{jacot2018neural} is a different model for infinitely wide neural networks. For two layer networks, it is obtained as the limiting object in a scaling regime for parameters which makes the Barron norm infinite. When training networks on empirical risk, in a certain scaling regime between the number of data points, the number of neurons, and the initialization of parameters, it can be shown that parameters do not move far from their initial position according to a parameter distribution $\bar \pi$, and that the gradient flow optimization of neural networks is close to the optimization of a kernel method for all times \cite{du2018gradient, weinan2019comparative}. This kernel is called the {\em neural tangent kernel} (NTK) and is obtained as the sum of derivatives of the feature function with respect to all trainable parameters. It linearizes the dynamics at the initial parameter distribution. For networks with one hidden layer this is
\begin{align*}
k(x,x') &= \int_{\R\times \R^d\times \R} \nabla_{(a,w,b)} \big(a\sigma(w^Tx+b)\big) \cdot \nabla_{(a,w,b)} \big(a\sigma(w^Tx'+b)\big) \,\bar\pi(\d a \otimes \d w \otimes \d b)\\
	&= \int_{\R\times \R^d\times \R} \sigma(w^Tx + b)\,\sigma(w^Tx'+ b) \\
		&\qquad+ a^2 \sigma'(w^Tx+b)\,\sigma'(w^Tx'+b) \big[\langle x,x'\rangle+1\big]\,\bar\pi(\d a \otimes \d w\otimes \d b)\\
	&= k_{RF}(x,x') + \int_{\R\times \R^d\times \R} a^2\,\sigma'(w^Tx+b)\,\sigma'(w^Tx'+b) (x,1) \,\cdot\,(x,1)\,\bar\pi(\d a \otimes \d w\otimes \d b)
\end{align*}
where $k_{RF}$ denotes the random feature kernel with distribution $P_{(w,b),\sharp}\bar\pi$. The second term is obtained on the right hand side is a positive definite kernel in itself. This can be seen most easily by recalling that
\begin{align*}
\sum_{i,j}c_i c_j & \int_{\R^{d+2}}\nabla_{(w,b)} \big(a\sigma(w^Tx_i+b)\big) \cdot \nabla_{(w,b)} \big(a\sigma(w^Tx_j+b)\big) \,\bar\pi(\d a \otimes \d w \otimes \d b)\\
	&=\int_{\R^{d+2}}\nabla_{(w,b)} a^2 \left(\sum_ic_i\,\sigma(w^Tx_i+b)\right) \cdot \nabla_{(w,b)} \left(\sum_jc_j\,\sigma(w^Tx_j+b)\right) \,\bar\pi(\d a \otimes \d w \otimes \d b)\\
	&= \int_{\R^{d+2}}\left|\nabla_{(w,b)} a^2 \left(\sum_ic_i\,\sigma(w^Tx_i+b)\right)\right|^2 \,\bar\pi(\d a \otimes \d w \otimes \d b)\\
	&\geq 0.
\end{align*}
On the other hand, if we assume that $|a| = a_0$ and $|(w,b)| = 1$ almost surely, we find that
\begin{align*}
\sum_{i,j}c_ic_j&\int_{\R\times \R^d\times \R} a^2\,\sigma'(w^Tx_i+b)\,\sigma'(w^Tx_j+b) \,(x_i,1) \,I\,(x_j,1)^T\,\bar\pi(\d a \otimes \d w\otimes \d b)\\
	&= \int_{\R\times \R^d\times \R} a^2 \left|\sum_i c_i\,\sigma'(w^Tx_i+b) \begin{pmatrix}x_i\\ 1\end{pmatrix}\right|^2 \bar\pi(\d a \otimes \d w\otimes \d b)\\
	&\leq a_0^2 \int_{\R\times \R^d\times \R} \left|\sum_i c_i\,\sigma'(w^Tx_i+b) \begin{pmatrix}x_i\\ 1\end{pmatrix} \cdot \begin{pmatrix} w\\ b\end{pmatrix} \right|^2 \bar\pi(\d a \otimes \d w\otimes \d b)\\
	&= a_0^2 \int_{\R\times \R^d\times \R} \left|\sum_i c_i\,\sigma'(w^Tx_i+b) (w^Tx_i+b) \right|^2 \bar\pi(\d a \otimes \d w\otimes \d b)\\
	&= a_0^2 \int_{\R\times \R^d\times \R} \left|\sum_i c_i\,\sigma(w^Tx_i+b) \right|^2 \bar\pi(\d a \otimes \d w\otimes \d b)
\end{align*}
since $\sigma$ is positively one-homogeneous. Thus the NTK satisfies
\[
k_{RF} \leq k \leq (1+ |a_0|^2)\,k_{RF}
\]
in the sense of quadratic forms. In particular, the eigenvalues of the NTK and the random feature kernel decay at the same rate. Clearly, in exchange for larger constants it suffices to assume that $(a, w, b)$ are bounded. In practice, the intialization of $(w,b)$ is Gaussian, which concentrates close to the Euclidean sphere of radius $\sqrt d$ in $d$ dimensions. 

The neural tangent kernel is also defined for deep networks, see for example \cite{arora2019exact, lee2019wide, yang2019scaling, du2018bgradient, weinan2019analysis}.
%

\newcommand{\etalchar}[1]{$^{#1}$}

\end{document}